\documentclass[11pt]{amsart}

\usepackage{verbatim,amsmath,amssymb,amsthm,amscd,color,graphics}
\usepackage{graphicx}
\usepackage[colorlinks=true, pdfstartview=FitV, linkcolor=blue, citecolor=blue, urlcolor=blue]{hyperref}

\usepackage{tikz,ytableau}

\numberwithin{equation}{section}
 
\newtheorem{theorem}{Theorem}
\newtheorem{proposition}[theorem]{Proposition}
\newtheorem{lemma}[theorem]{Lemma}
\newtheorem{corollary}[theorem]{Corollary}

\numberwithin{theorem}{section}
 
\theoremstyle{definition}
\newtheorem{definition}[theorem]{Definition}
\newtheorem{remark}[theorem]{Remark}
\newtheorem{example}[theorem]{Example}
 
\usepackage[letterpaper,margin=1.25in]{geometry}

\definecolor{darkred}{rgb}{0.7,0,0} 
\definecolor{darkgreen}{rgb}{0, .6, 0} 
\newcommand{\defn}[1]{{\color{darkred}\emph{#1}}} 

\usepackage[colorinlistoftodos]{todonotes}

\renewcommand{\L}{\mathsf{L}}
\newcommand{\R}{\mathsf{R}}
\newcommand{\De}{\mathsf{D}}
\newcommand{\ZZ}{\mathbb{Z}}  
\newcommand{\mm}{\mathsf{minimaj}} 
\newcommand{\maj}{\mathsf{maj}}
\newcommand{\wt}{\mathsf{wt}}
\renewcommand\read{\mathsf{read}}
\newcommand\calOP{\mathcal{OP}}

\newcommand\encircle[1]{%
  \tikz[baseline=(X.base)] 
    \node (X) [draw, shape=circle, inner sep=0] {\strut $#1$};}
           
\begin{document}

\title{A minimaj-preserving crystal on ordered multiset partitions}

\author[Benkart]{Georgia Benkart}
\address[G. Benkart]{Department of Mathematics, University of Wisconsin-Madison, 480 Lincoln Dr.
Madison, WI 53706, U.S.A.}
\email{benkart@math.wisc.edu}

\author[Colmenarejo]{Laura Colmenarejo}
\address[L. Colmenarejo]{Department of Mathematics and Statistics,  York University, 4700 Keele Street, Toronto, 
Ontario M3J 1P3, Canada}
\email{lcolme@yorku.ca}
\urladdr{http://www.yorku.ca/~lcolme/}

\author[Harris]{Pamela E. Harris}
\address[P. E. Harris]{Mathematics and Statistics, Williams College, Bascom House, Rm 106C
Bascom House, 33 Stetson Court, Williamstown, MA 01267, U.S.A.} 
\email{peh2@williams.edu}
\urladdr{https://math.williams.edu/profile/peh2/}

\author[Orellana]{Rosa Orellana}
\address[R. Orellana]{Mathematics Department, Dartmouth College, 6188 Kemeny Hall,
Hanover, NH 03755, U.S.A.}
\email{Rosa.C.Orellana@dartmouth.edu}
\urladdr{https://math.dartmouth.edu/~orellana/}

\author[Panova]{Greta Panova}
\address[G. Panova]{Mathematics Department, University of Pennsylvania, 
209 South 33rd St, Philadelphia, PA 19104, U.S.A}
\email{panova@math.upenn.edu}
\urladdr{https://www.math.upenn.edu/~panova/}

\author[Schilling]{Anne Schilling}
\address[A. Schilling]{Department of Mathematics, University of California, One Shields
Avenue, Davis, CA 95616-8633, U.S.A.}
\email{anne@math.ucdavis.edu}
\urladdr{http://www.math.ucdavis.edu/\~{}anne}

\author[Yip]{Martha Yip}
\address[M. Yip]{Department of Mathematics, University of Kentucky, 715 Patterson Office Tower,
Lexington, KY 40506-0027, U.S.A.}
\email{martha.yip@uky.edu}
\urladdr{http://www.ms.uky.edu/~myip/}

\date{\today}
\keywords{Delta Conjecture, ordered multiset partitions, minimaj statistic, crystal bases, equidistribution of statistics}
\subjclass[2000]{Primary 05A19; Secondary 05A18, 05E05, 05E10, 20G42, 17B37}

\begin{abstract}
We provide a crystal structure on the set of ordered multiset partitions, which recently arose in
the pursuit of the Delta Conjecture.  This conjecture was stated by Haglund, Remmel and Wilson as a generalization 
of the Shuffle Conjecture. Various statistics on ordered multiset partitions arise in the combinatorial
analysis of the Delta Conjecture, one of them being the minimaj statistic, which is a variant of the
major index statistic on words. Our crystal has the property that the minimaj statistic is constant on connected 
components of the crystal. In particular, this yields another proof of the Schur positivity of the graded Frobenius 
series of the generalization $R_{n,k}$ due to Haglund, Rhoades and Shimozono of the coinvariant algebra $R_n$.
The crystal structure also enables us to demonstrate the equidistributivity of the minimaj statistic with the major index 
statistic on ordered multiset partitions.
\end{abstract}

\maketitle   

\section{Introduction}

The Shuffle Conjecture~\cite{HHLRU.2005}, now a theorem due to Carlsson and Mellit~\cite{CM.2015},
provides an explicit combinatorial description of the bigraded Frobenius characteristic of the $S_n$-module of 
diagonal harmonic polynomials. It is stated in terms of parking functions and involves two statistics, $\mathsf{area}$ 
and $\mathsf{dinv}$.

Recently, Haglund, Remmel and Wilson~\cite{HRW.2015} introduced a generalization of the Shuffle Theorem,
coined the Delta Conjecture. The Delta Conjecture involves two quasisymmetric functions
$\mathsf{Rise}_{n,k}(\mathbf{x};q,t)$ and $\mathsf{Val}_{n,k}(\mathbf{x};q,t)$, which have 
combinatorial expressions in terms of labelled Dyck paths. In this paper, we are only concerned with the specializations 
$q=0$ or $t=0$, in which case~\cite[Theorem 4.1]{HRW.2015} and~\cite[Theorem 1.3]{Rhoades.2016} show
\[
	\mathsf{Rise}_{n,k}(\mathbf{x};0,t) = \mathsf{Rise}_{n,k}(\mathbf{x};t,0) = 
	\mathsf{Val}_{n,k}(\mathbf{x};0,t) = \mathsf{Val}_{n,k}(\mathbf{x};t,0).
\]
It was proven in~\cite[Proposition 4.1]{HRW.2015} that
\begin{equation}
\label{equation.val}
	\mathsf{Val}_{n,k}(\mathbf{x};0,t) = \sum_{\pi\in \mathcal{OP}_{n,k+1}}  t^{\mm(\pi)} \mathbf{x}^{\wt(\pi)},
\end{equation}
where $\mathcal{OP}_{n,k+1}$ is the set of ordered multiset partitions of the multiset $\{1^{\nu_1},2^{\nu_2},\ldots\}$ 
into $k+1$ nonempty blocks and $\nu=(\nu_1,\nu_2,\ldots)$ ranges over all weak compositions of $n$.
The weak composition $\nu$ is also called the weight of $\pi$, denoted $\wt(\pi)=\nu$. 
In addition, $\mm(\pi)$ is the minimum value of the major index of the set partition $\pi$ over all possible ways to 
order the elements in each block of $\pi$. The symmetric function $\mathsf{Val}_{n,k}(\mathbf{x};0,t)$ is 
known~\cite{Wilson.2016,Rhoades.2016} to be Schur positive, meaning that the coefficients are polynomials in $t$ 
with nonnegative coefficients.

In this paper, we provide a crystal structure on the set of ordered multiset partitions $\mathcal{OP}_{n,k}$.
Crystal bases are $q\to 0$ shadows of representations for quantum groups $U_q(\mathfrak g)$~\cite{Kashiwara.1990,
Kashiwara.1991}, though they can also be understood from a purely combinatorial 
perspective~\cite{Stembridge.2003,Bump.Schilling.2017}. In type $A$, the character of a connected crystal
component with highest weight element of highest weight $\lambda$ is the Schur function $\mathsf{s}_\lambda$.
Hence, having a type $A$ crystal structure on a combinatorial set (in our case on $\mathcal{OP}_{n,k}$)
naturally yields the Schur expansion of the associated symmetric function. Furthermore, if the statistic
(in our case $\mm$) is constant on connected components, then the graded character can also be naturally
computed using the crystal.

Haglund, Rhoades and Shimozono~\cite{HRS.2016} introduced a generalization $R_{n,k}$ for $k\leqslant n$ 
of the coinvariant algebra $R_n$, with $R_{n,n}=R_n$. Just as the combinatorics of $R_n$ is governed by permutations 
in $S_n$, the combinatorics of $R_{n,k}$ is controlled by ordered set partitions of $\{1,2\ldots,n\}$ with $k$ blocks.
The graded Frobenius series of $R_{n,k}$ is (up to a minor twist) equal to $\mathsf{Val}_{n,k}(\mathbf{x};0,t)$.
It is still an open problem to find a bigraded $S_n$-module whose Frobenius image is  
$\mathsf{Val}_{n,k}(\mathbf{x};q,t)$. Our crystal provides another representation-theoretic interpretation of
$\mathsf{Val}_{n,k}(\mathbf{x};0,t)$ as a crystal character.

Wilson~\cite{Wilson.2016} analyzed various statistics on ordered multiset partitions, including $\mathsf{inv}$,
$\mathsf{dinv}$, $\maj$, and $\mm$. In particular, he gave a Carlitz type bijection, which
proves equidistributivity of $\mathsf{inv}$, $\mathsf{dinv}$, $\maj$ on $\mathcal{OP}_{n,k}$.
Rhoades~\cite{Rhoades.2016} provided a non-bijective proof that these statistics are also equidistributed with
$\mm$. Using our new crystal, we can give a bijective proof of the equidistributivity of the 
$\mm$ statistic and the $\maj$ statistic on ordered multiset partitions.

The paper is organized as follows. In Section~\ref{section.minimaj} we define ordered multiset partitions and 
the $\mm$ and $\maj$ statistics on them. In Section~\ref{section.bijection} we provide a bijection 
$\varphi$ from ordered multiset partitions to tuples of semistandard Young tableaux that will be used in 
Section~\ref{section.crystal} to define a crystal structure, which preserves $\mm$. We conclude in 
Section~\ref{section.equi} with a proof that the $\mm$ and $\maj$ statistics are equidistributed using 
the same bijection $\varphi$.

\subsection*{Acknowledgments}
Our work on this group project began at the workshop \emph{Algebraic Combinatorixx 2} at  the Banff International 
Research Station (BIRS) in May 2017.  ``Team Schilling,''  as our group of authors is known, would like to extend thanks 
to the organizers of ACxx2,  to BIRS for hosting this workshop, and to the Mathematical Sciences Research Institute (MSRI) 
for sponsoring a follow-up meeting of some of the group members at MSRI in July 2017 supported by the National Science
Foundation under Grant No. DMS-1440140.
We would like to thank Meesue Yoo for early collaboration and Jim Haglund, Brendon Rhoades and Andrew Wilson for 
fruitful discussions. This work benefited from computations and experimentations in {\sc Sage}~\cite{combinat,sage}.

P. E. Harris was partially supported by NSF grant DMS--1620202. 
R. Orellana was partially supported by NSF grant DMS--1700058. 
G. Panova was partially supported by NSF grant DMS--1500834.
A. Schilling was partially supported by NSF grant  DMS--1500050. 
M. Yip was partially supported by Simons Collaboration grant 429920.

\section{Ordered multiset partitions and the minimaj and maj statistics}
\label{section.minimaj}

We consider \defn{ordered multiset partitions} of order $n$ with $k$ blocks.  
Given a weak composition $\nu = (\nu_1, \nu_2, \ldots)$ of $n$ into nonnegative integer parts,  
which we denote $\nu \models n$,  let $\mathcal{OP}_{\nu,k}$ be the set of partitions of the multiset 
$\{i^{\nu_i} \mid i \geqslant 1\}$  into $k$ nonempty ordered blocks, such that the elements within each block are distinct.  
For each $i\geqslant 1$, the notation $i^{\nu_i}$ should be interpreted as saying that the integer $i$ occurs $\nu_i$ times 
in such a partition.  The weak composition $\nu$ is also called the \defn{weight} $\wt(\pi)$ of 
$\pi \in \mathcal{OP}_{\nu,k}$. Let
\[
	\mathcal{OP}_{n,k}  =  \bigcup_{\nu \models n}  \mathcal{OP}_{\nu,k}.
\]
It should be noted that in the literature $\mathcal{OP}_{n,k}$ is sometimes used for ordered set partitions rather
than ordered multiset partitions (that is, without letter multiplicities).

We now specify a particular reading order for an ordered multiset partition $\pi = (\pi_1\mid \pi_2 \mid \ldots \mid \pi_k)  
\in \mathcal{OP}_{n,k}$ with blocks $\pi_i$. Start by writing $\pi_k$ in increasing order. Assume $\pi_{i+1}$ has been ordered, 
and  let $r_i$ be the largest integer in $\pi_i$ that is less than or equal to the leftmost element of $\pi_{i+1}$.   If no such $r_i$ 
exists, arrange $\pi_i$ in increasing order.  When such an $r_i$ exists, arrange the elements of $\pi_i$ in increasing order, 
and then cycle them so that $r_i$ is the rightmost number. Continue with $\pi_{i-1}, \dots, \pi_2, \pi_1$ until all blocks have 
been ordered. This ordering of the numbers in $\pi$ is defined in \cite{HRW.2015} and is called the \defn{minimaj order}.  

\begin{example}
\label{example.pi}
lf $\pi = (157 \mid 24 \mid 56 \mid 468 \mid 13 \mid 123) \in \mathcal{OP}_{15,6}$,  then the minimaj order of $\pi$ is  
$\pi = (571 \mid 24 \mid 56 \mid 468 \mid 31 \mid 123)$.
\end{example}

For two sequences $\alpha,\beta$ of integers,  we write $\alpha < \beta$ to mean that each element of $\alpha$ is less
than every element of $\beta$.  Suppose $\pi \in \mathcal{OP}_{n,k}$ is in minimaj order. Then each block $\pi_i$ of $\pi$ 
is nonempty and 
can be written in the form $\pi_i = b_i \alpha_i \beta_i$, where $b_i \in \ZZ_{>0}$, and $\alpha_i,\beta_i$ are sequences 
(possibly empty) of distinct increasing integers such that either $\beta_i < b_i  < \alpha_i$ or $\alpha_i=\emptyset$. 
Inequalities with empty sets should be ignored. 

\begin{lemma}
\label{lemma.minimaj order}
With the above notation, $\pi \in \mathcal{OP}_{n,k}$ is in minimaj order if the following hold:
\begin{itemize}
\item [{\rm (1)}]  \ $\pi_k = b_k\alpha_k$ with $b_k<\alpha_k$ and $\beta_k = \emptyset$;
\item [{\rm (2)}]  \  for $1 \leqslant i < k$,    either
\begin{itemize} \item[{(a)}] \ $\alpha_i = \emptyset$,  $\pi_i = b_i \beta_i$,  and  
$b_i < \beta_i \leqslant b_{i+1}$,    or
 \item[{(b)}] \   $\beta_i \leqslant b_{i+1} < b_i < \alpha_i$.  
 \end{itemize}
 \end{itemize}
\end{lemma}
 
A sequence or word $w_1 w_2 \cdots w_n$ has a \defn{descent} in position $1\leqslant i<n$ if $w_i>w_{i+1}$.
Let $\pi \in \mathcal{OP}_{n,k}$ be in minimaj order. Observe that a descent occurs in $\pi_i$ only in Case~2\,(b)
of Lemma~\ref{lemma.minimaj order}, and such a descent is either between the largest and smallest elements of $\pi_i$
or between the last element of $\pi_i$ and the first element of $\pi_{i+1}$.   
 
\begin{example}
Continuing Example~\ref{example.pi} with $\pi = (571 \mid 24 \mid 56 \mid 468 \mid 31 \mid 123)$, we have
\[
 	\begin{array}{lll}
	b_1 = 5, \alpha_1 =7, \beta_1 = 1 &\qquad b_2 = 2, \alpha_2 = \emptyset, \beta_2 =4 &\qquad
 	b_3 = 5, \alpha_3 = 6, \beta_3 =\emptyset \\ 
	b_4 = 4,  \alpha_4 = 68, \beta_4 =  \emptyset  &\qquad
 	b_5 = 3,  \alpha_5 = \emptyset, \beta_5 = 1 &\qquad b_6 = 1,  \alpha_6 = 23,  \beta_6 = \emptyset. 
	\end{array}
\]
\end{example}
 
Suppose that $\pi$ in minimaj order has descents in positions
\[
	\De(\pi) = \{d_1, d_1+d_2, \ldots, d_1+d_2 + \cdots + d_\ell\}
\]
for some $\ell \in [0,k-1]$  ($\ell= 0$ indicates no descents). Furthermore assume that these descents occur in the blocks  
$\pi_{i_1}, \pi_{i_1+i_2}, \ldots,\pi_{i_1+i_2+\cdots + i_\ell}$, where  $i_j >0$ for $1\leqslant j \leqslant \ell$ and 
$i_1+i_2+\cdots+i_\ell <k$.  Assume $d_{\ell+1}$ and $i_{\ell+1}$ are the distances to the end, that is,  
$d_1+d_2 + \cdots+d_\ell + d_{\ell+1} = n$ and $i_1+i_2+\cdots+ i_\ell + i_{\ell+1} = k$.  

The \defn{minimaj statistic} $\mm(\pi)$ of $\pi \in \mathcal{OP}_{n,k}$ as given by~\cite{HRW.2015} is
\begin{equation}
\label{equation.minimaj}
	\mm(\pi) = \sum_{d \in \De(\pi)} d  =  \sum_{j=1}^\ell (\ell+1-j) d_j.
\end{equation}

\begin{example}
The descents for the multiset partition  $\pi = (57.1 \mid 24 \mid 56. \mid 468. \mid 3.1 \mid 123)$  occur at positions 
$\De(\pi)=\{2,7,10,11\}$ and are designated with periods. Hence $\ell=4$, $d_1 = 2$, $d_2 = 5$, $d_3 = 3$, $d_4 = 1$ 
and $d_5 = 4$, and  $\mm(\pi) = 2 + 7 + 10 + 11 = 30$.  The descents occur in blocks \  $\pi_1$, $\pi_3$, $\pi_4$, and
$\pi_5$,  so that $i_1 = 1$, $i_2 = 2$, $i_3 = 1$, $i_4 = 1$, and $i_5 = 1$.  
\end{example}

To define the \defn{major index} of $\pi \in \mathcal{OP}_{n,k}$, we consider the word $w$ obtained by ordering each 
block $\pi_i$ in decreasing order, called the \defn{major index order}~\cite{Wilson.2016}.
Recursively construct a word $v$ by setting $v_0=0$ and $v_j = v_{j-1}+ \chi(j \text{ is the last position in its block})$
for each $1\leqslant j \leqslant n$. Here $\chi(\text{True})=1$ and $\chi(\text{False})=0$. Then
\begin{equation}
\label{equation.maj}
	\maj(\pi) = \sum_{j \colon w_j>w_{j+1}} v_j.
\end{equation}

\begin{example}
Continuing Example~\ref{example.pi}, note that the major index order of $\pi = (157 \mid 24 \mid 56 \mid 468 \mid 13 
\mid 123) \in \mathcal{OP}_{15,6}$ is $\pi = (751 \mid 42 \mid 65 \mid 864 \mid 31 \mid 321)$. Writing the word $v$ 
underneath $w$ (omitting $v_0=0$), we obtain
\begin{equation*}
\begin{split}
	w &= 751 \mid 42 \mid 65 \mid 864\mid 31 \mid 321\\
	v &= 001 \mid 12 \mid 23 \mid 334 \mid 45 \mid 556,
\end{split}
\end{equation*}
so that $\maj(\pi) = 0+0+1+2+3+3+4+4+5+5=27$.
\end{example}

Note that throughout this section, we could have also restricted ourselves to ordered multiset partitions with 
letters in $\{1,2,\ldots, r\}$ instead of $\ZZ_{>0}$. That is, let $\nu=(\nu_1,\ldots,\nu_r)$ be a weak composition of $n$ and let
$\mathcal{OP}^{(r)}_{\nu,k}$ be the set of partitions of the multiset $\{i^{\nu_i} \mid 1\leqslant i \leqslant r\}$ into
$k$ nonempty ordered blocks, such that the elements within each block are distinct.  Let
\[
	\mathcal{OP}^{(r)}_{n,k}  =  \bigcup_{\nu \models n}  \mathcal{OP}^{(r)}_{\nu,k}.
\]
This restriction will be important when we discuss the crystal structure on ordered multiset partitions.

\section{Bijection with tuples of semistandard Young tableaux}
\label{section.bijection}

In this section, we describe a bijection from ordered multiset partitions to tuples of semistandard Young tableaux
that allows us to impose a crystal structure on the set of ordered multiset partitions in Section~\ref{section.crystal}.

Recall that a \defn{semistandard Young tableau} $T$ is a filling of a (skew) Young diagram (also called the \defn{shape} of $T$)
with positive integers that weakly increase across rows and strictly increase down columns. The \defn{weight} of $T$ is 
the tuple $\wt(T)=(a_1,a_2,\ldots)$, where $a_i$ records the number of letters $i$ in $T$. The set of semistandard Young 
tableaux of shape $\lambda$, where $\lambda$ is a (skew) partition, is denoted by $\mathsf{SSYT}(\lambda)$. If we want to 
restrict the entries in the semistandard Young tableau from $\ZZ_{>0}$ to a finite alphabet $\{1,2,\ldots,r\}$, we denote the set by $\mathsf{SSYT}^{(r)}(\lambda)$.  

The tableaux relevant for us here are of two types: a single column of boxes with entries that increase from top to bottom,  
or a skew ribbon tableau.  If $\gamma =(\gamma_1,\gamma_2, \dots, \gamma_m)$ is a skew ribbon shape with 
$\gamma_j$ boxes in the $j$-th row starting from the bottom, the ribbon condition requires that row $j+1$ starts in the 
last column of row $j$. This condition is equivalent to saying that $\gamma$ is connected and contains no $2 \times 2$ 
block of squares.  For example
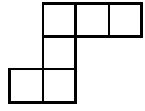
\ytableausetup{boxsize=1.1em}
\[
 \ytableaushort{\none\none {\mbox{}}{\mbox{}}{\mbox{}},\none\none {\mbox{}},\none 
	{\mbox{}}{\mbox{}}}
\]
corresponds to $\gamma = (2,1,3)$.
Let $\mathsf{SSYT}(1^c)$ be the set of semistandard Young tableaux obtained by filling a column of 
length $c$ and $\mathsf{SSYT}(\gamma)$ be the set of semistandard Young tableaux obtained by filling the skew ribbon
shape $\gamma$.

To state our bijection, we need the following notation. For fixed positive integers $n$ and $k$, assume 
$\mathrm{D}= \{d_1,d_1+d_2,\ldots,d_1+d_2+\cdots+d_\ell\} \subseteq \{1,2,\dots,n-1\}$ and 
$\mathrm{I} = \{i_1,i_1+i_2,\ldots,i_1+i_2+\cdots+i_\ell\} \subseteq\{1,2,\dots, k-1\}$ are sets of $\ell$ distinct elements 
each. Define $d_{\ell+1} := n-(d_1+\cdots +d_\ell)$ and $i_{\ell+1} := k - (i_1+\cdots +i_\ell)$.

\begin{proposition}
\label{P:biject}  
For fixed positive integers $n$ and $k$ and sets $\mathrm{D}$ and $\mathrm{I}$ as above, let
\[
	\mathrm{M(D,I)} = \{\pi \in \mathcal{OP}_{n,k} \mid  \De(\pi) = \mathrm{D}, \ 
	\text{and the descents occur in $\pi_i$ for $i \in \mathrm{I}$}\}.
\]
Then the following map is a weight-preserving bijection:
\begin{equation}
\begin{split}
	\varphi \colon \mathrm{M(D,I)} &\rightarrow \mathsf{SSYT}(1^{c_1}) \times \cdots 
	\times \mathsf{SSYT}(1^{c_\ell}) \times \mathsf{SSYT}(\gamma)\\
	\pi &\mapsto T_1 \times \cdots \times T_\ell \times T_{\ell+1}
\end{split}
\end{equation}
where
\begin{itemize}  
\item [{\rm (i)}] \ $\gamma = (1^{d_1-i_1}, i_1, i_2, \dots, i_{\ell+1})$ and $c_j = d_{\ell+2-j} - i_{\ell+2-j}$ for 
$1\leqslant j \leqslant \ell$.
\item [{\rm(ii)}] \ The skew ribbon tableau $T_{\ell+1}$ of shape $\gamma$ is constructed as follows:
\begin{itemize}  
\item [$\bullet$] The entries in the first column of the skew ribbon tableau $T_{\ell+1}$ beneath the first box are the first 
$d_1-i_1$ elements of $\pi$ in increasing order from top to bottom, excluding any $b_j$ in that range.
\item[$\bullet$] The remaining rows $d_1-i_1+j$ of $T_{\ell+1}$ for $1\leqslant j \leqslant \ell+1$ are filled with\newline
$b_{i_1+\dots + i_{j-1} + 1}, b_{i_1+\dots + i_{j-1} + 2}, \dots, b_{i_1+\dots + i_j}$.
\end{itemize} 
\item [{\rm(iii)}] The tableau $T_j$ for $1\leqslant j \leqslant \ell$ is the column filled with the elements of $\pi$ from the positions 
$d_1+d_2+\cdots + d_{\ell-j+1}+1$ through and including position $d_1+d_2+\cdots + d_{\ell-j+2}$, but excluding any 
$b_i$ in that range.
\end{itemize}
 \end{proposition}

 Note that in item (ii), the rows of $\gamma$ are assumed to be numbered from bottom 
 to top and are filled starting with row  $d_1-i_1+1$ and ending with row $d_1-i_1+\ell+1$
at the top.
  
Also observe that since the bijection stated in Proposition~\ref{P:biject} preserves the weight, it can be restricted to a 
bijection
 \[
 	\varphi \colon \mathrm{M(D,I)}^{(r)} \rightarrow \mathsf{SSYT}^{(r)}(1^{c_1}) \times \cdots 
	\times \mathsf{SSYT}^{(r)}(1^{c_\ell}) \times \mathsf{SSYT}^{(r)}(\gamma),
\]
where $\mathrm{M(D,I)}^{(r)} = \mathrm{M(D,I)} \cap \mathcal{OP}^{(r)}_{n,k}$.

Before giving the proof, it is helpful to consider two examples to illustrate the map $\varphi$.  
 
\begin{example}\label{ex:ell=0}  
 When the entries of $\pi \in \mathcal{OP}_{n,k}$ in minimaj order are increasing, then $\ell = 0$.   
 In this case, $d_1 = n$ and $i_1 = k$.   The mapping $\varphi$ takes $\pi$ to the semistandard tableau $T =T_1$ 
 that is of ribbon-shape $\gamma = (1^{n-k},k)$. The  entries of the boxes in the first column of the tableau $T$ are  
 $b_1$, followed by the $n-k$ numbers in the sequences $\beta_1,\beta_2,\dots, \beta_{k-1},\alpha_k$
from top to bottom. (The fact that $\pi$ has no descents means that all the $\alpha_i = \emptyset$ for $1\leqslant i <k$
and we are in Case 2\,(a) of Lemma~\ref{lemma.minimaj order} for $1\leqslant i<k$ and Case 1 for $i=k$.)  Columns $2$ 
through $k$ of $T_1$ are filled with the numbers $b_2,\dots,b_k$ respectively, and $b_2 \leqslant b_3 \leqslant \cdots 
\leqslant b_k$.  The result is a semistandard tableau $T_1$ of hook shape. 

For example, consider $\pi = (12 \mid 2 \mid 234) \in \mathcal{OP}_{6,3}$. Then $\gamma=(1^3,3)$ and
\[
	T_1 = \ytableaushort{122,2,3,4}\;.
\]

Now suppose that $T$ is such a hook-shape tableau with entries $b_1,b_2,\dots,b_{k}$ from left to right in its top row, 
and entries $b_1, t_1, \dots, t_{n-k}$ down its first column. The inverse $\varphi^{-1}$  maps $T$ to the set partition $\pi$ 
that has as its first block $\pi_1= b_1\beta_1$, where $\beta_1=t_1, \dots, t_{m_1}$, and $t_1 < \dots < t_{m_1} \leqslant b_2$,
but $t_{m_1+1} > b_2$ so that $\beta_1$ is in the interval $(b_1,b_2]$. The second block of $\pi$ is given by 
$\pi_2 = b_2 \beta_2$, where $\beta_2 = t_{m_1+1},\dots,t_{m_2}$, and $t_{m_1+1}< t_{m_1+2}< \dots <t_{m_2} \leqslant
b_3$,  but $t_{m_2+1} > b_3$ and $\beta_2  \subseteq (b_2,b_3]$.  
Continuing in this fashion, we set $\pi_k = b_k \alpha_k$, where $\alpha_k = t_{m_{k-1}+1},\dots, t_{n-k}$ and
$\alpha_k \subseteq (b_k,+\infty)$.   Then  $\varphi^{-1}(T) = \pi = (\pi_1\mid \pi_2 \mid \cdots \mid \pi_k)$,
where the ordered multiset partition $\pi$ has no descents.  
\end{example} 
 
\begin{example}
\label{example.pi phi}
The ordered multiset partition $\pi = (124 \mid 45. \mid 3 \mid 46.1\mid 23.1\mid 1 \mid 25) \in \mathcal{OP}_{15,7}$ 
has the following data:
 \[
 	\begin{array}{lll}  b_1 = 1,  \alpha_1 =\emptyset, \beta_1 = 24  & \qquad b_2 = 4,\alpha_2 = 5,
	\beta_2 =\emptyset  & \qquad b_3 = 3, \alpha_3 = \emptyset, \beta_3 =\emptyset \\ 
	b_4 = 4, \alpha_4 = 6, \beta_4 =  1 & \qquad  b_5 = 2, \alpha_5 = 3,  \beta_5 = 1 &  \qquad b_6 = 1,
 	\alpha_6 = \emptyset,  \beta_6 = \emptyset \\
	b_7=2, \alpha_7 = 5, \beta_7 = \emptyset & \\
  	\end{array}
\]
and $\ell=3$, $d_1=5, d_2=d_3=3, d_4=4$ and $i_1=i_2=2, i_3=1, i_4=2$. Then 
\ytableausetup{boxsize=1.1em}
\[
	\pi = ({\color{blue}1}{\color{red}2}{\color{red}4} \mid {\color{blue}4}{\color{red}5}. \mid {\color{blue}3} \mid 
	{\color{blue}4}{\color{darkgreen}6}.{\color{orange}1}\mid {\color{blue}2}{\color{orange}3}.{\color{violet}1}\mid {\color{blue}1} 
	\mid {\color{blue}2}{\color{violet}5}) \mapsto
	\ytableaushort{{\color{violet}1},{\color{violet}5}} \times  \ytableaushort{{\color{orange}1},{\color{orange}3}} 
	\times \ytableaushort{{\color{darkgreen}6}} \times 
	\ytableaushort{\none\none {\color{blue}1}{\color{blue}2},\none\none {\color{blue}2},\none 
	{\color{blue}3}{\color{blue}4},{\color{blue}1}{\color{blue}4},{\color{red}2},{\color{red}4},{\color{red}5}}\;.
\]
\end{example}

It is helpful to keep the following picture in mind during the proof of Proposition~\ref{P:biject}, where the map $\varphi$ 
is taking the ordered multiset partition $\pi$ to the collection of tableaux $T_i$  as illustrated below. We adopt the shorthand 
notation $\eta_j :=  i_1+\cdots+i_j$ for $1\leqslant j \leqslant \ell$, where we also set $\eta_0=0$ and $\eta_{\ell+1}=k$:
\[
	\pi = (b_1 \beta_1 | b_2 \beta_2 | \cdots |b_{\eta_1} \alpha_{\eta_1}. \beta_{\eta_1}|b_{\eta_1+1} \beta_{\eta_1+1} 
	|\cdots  |b_{\eta_j} \alpha_{\eta_j}.\beta_{\eta_j}| b_{\eta_j+1} \beta_{\eta_j+1} | \cdots | b_k \alpha_k)
\]
\ytableausetup{boxsize=2.9em}
\begin{equation}
\label{equation.T picture}
	T_{\ell+1-j} = \ytableaushort{{\beta_{\eta_j}},{\beta_{\eta_j+1}},{\vdots},{\beta_{\eta_{j+1}-1}},{\alpha_{\eta_{j+1}}}} 
	\; \text{ for }1\leqslant j\leqslant \ell, \quad
	T_{\ell+1} = \ytableaushort{\none\none\none\none{ b_{\eta_\ell+1}}{\cdots}{b_{\eta_{\ell+1}}},
	\none\none\none\none{\vdots},
	\none\none{b_{\eta_{j-1}+1}}{\cdots}
	{b_{\eta_j}}, \none\none{\vdots},{b_1}{\cdots}{b_{\eta_1}},{\beta_1},{\vdots},{\beta_{\eta_1-1}},{\alpha_{\eta_1}}}\;.
\end{equation}

\begin{proof}[Proof of Proposition \ref{P:biject}]
Since the entries of $\pi$ are mapped bijectively to the entries of $T_1 \times T_2 \times \cdots \times T_{\ell+1}$,  
the map $\varphi$ preserves the total weight $\wt(\pi)=(p_1,p_2,\ldots) \mapsto \, \wt(T)$, where $p_i$ is the number of 
entries $i$ in $\pi$ for $i \in \mathbb{Z}_{>0}$. We need to show that $\varphi$ is well defined and 
exhibit its inverse. For this, we can assume that $\ell \geqslant 1$, as the case $\ell = 0$ was treated in
Example~\ref{ex:ell=0}. 

Observe first that there are $d_j$ entries in $\pi$ which are between two consecutive descents, 
and among these entries there are exactly $i_j$ entries that are first elements of a block, since descents happen $i_j$
blocks apart. This implies that the tableaux have the shapes claimed. 

To see that the tableaux are semistandard, consider first  $T_{\ell+1}$, and let $\eta_j = i_1+\cdots+i_j$ as above. 
A row numbered $d_1-i_1+j$ for $1 \leqslant j \leqslant \ell+1$ is weakly increasing, because the lack of a descent in 
a block $\pi_i$ means $b_i \leqslant b_{i+1}$, and this holds for $i$ in the interval  $\eta_{j-1} +1, \ldots, \eta_{j}$ between 
two consecutive descents.  The leftmost column is strictly increasing because it consists of the elements $b_1 < \beta_1 
< \beta_2 < \cdots <\beta_{\eta_1-1} <\alpha_{\eta_1}$ (the lack of a descent before $\pi_{\eta_1}$ implies that 
$\alpha_i=\emptyset$ for $i<\eta_1$ and $b_i <\beta_i \leqslant b_{i+1}< \beta_{i+1}$ by  Case 2\,(a)
of Lemma~\ref{lemma.minimaj order}).

The rest of the columns of $T_{\ell+1}$ contain elements $b_i$, where $b_{\eta_{j-1}+1}$ is the first element in row 
$d_{1}-i_1+j$ and $b_{\eta_j}$ is the last, and $b_{\eta_{j}+1}$ is the first element in the row immediately above it.  We have 
$b_{\eta_j} > b_{\eta_j+1}$,  since there is a descent in block $\pi_{i_j}$ which implies this inequality by the ordering 
condition in Case 2\,(b) of Lemma~\ref{lemma.minimaj order}.

The strict inequalities for the column tableaux $T_1,\ldots,T_{\ell}$ hold for the same reason that they hold for the first 
column in $T_{\ell+1}$. That is,  the columns consist of the elements $\beta_{\eta_j} < \beta_{\eta_j+1} <\cdots <
\beta_{\eta_{j+1}-1} <\alpha_{\eta_{j+1}}$, where all the $\alpha_i$ for $\eta_j\leqslant i<\eta_{j+1}$ are in fact 
$\emptyset$, since we are in Case~2\,(a) of Lemma~\ref{lemma.minimaj order} here. 

Next, to show that $\varphi$ is a bijection, we describe the inverse map of $\varphi$.
For $\mathrm{D} = \{d_1,d_1+d_2, \ldots, d_1+d_2+\cdots+d_\ell\} \subseteq \{1,2,\dots,n-1\}$ and $\mathrm{I}
= \{i_1,i_1+i_2,\dots,i_1+i_2+\cdots+i_\ell\} = \{\eta_1,\eta_2,\ldots,\eta_\ell \} \subseteq \{1,2,\dots,k-1\}$ with $\ell$ distinct
elements each, suppose $d_{\ell+1}$ and $i_{\ell+1}$ are such that $d_1+d_2+\cdots+d_{\ell+1} = n$ and 
$\eta_{\ell+1}=i_1+i_2+\cdots+i_{\ell+1}=k$.  Assume
$T_1 \times \cdots \times T_\ell \times T_{\ell+1} \in \mathsf{SSYT}(1^{c_1}) \times \cdots \times \mathsf{SSYT}(1^{c_\ell}) 
\times \mathsf{SSYT}(\gamma)$, where $\gamma = (1^{d_1-i_1},i_2,\dots,i_{\ell+1})$ and $c_j = d_{\ell+2-j} - i_{\ell+2-j}$ 
for $1\leqslant j \leqslant \ell$. We construct $\pi$ by applying the following algorithm.

Read off the bottom $d_1-i_1$ entries of the first column of $T_{\ell+1}$. Let $b_1$ be the element immediately above 
these entries in the first column of $T_{\ell+1}$, and note that $b_1$ is less than all of them. Let $b_2,\dots, b_{i_1}$ be the 
elements in the same row of $T_{\ell+1}$ as $b_1$, reading from left to right.  Assign  $b_{\eta_1+1},\ldots, b_{\eta_2}$ to the 
elements in the next higher row, and so forth, until reaching row $d_1-i_1+\ell+1$ (the top row) of $T_{\ell+1}$ 
and assigning $b_{\eta_\ell+1},\dots, b_{\eta_{\ell+1}}=b_k$ to its entries.
The elements in $\beta_1,\ldots,\beta_{\eta_1-1},\alpha_{\eta_1}$ are obtained by cutting the entries in 
the first column of $T_{\ell+1}$ above $b_1$, so that $\beta_i$ lies in the interval $(b_i, b_{i+1}]$, and  $\alpha_{\eta_1}$ 
lies in the interval  $(b_{\eta_1},\infty)$.

Now for $1\leqslant j\leqslant \ell$, we obtain $\beta_{\eta_j},\beta_{\eta_j+1},\ldots,\beta_{\eta_{j+1}-1},\alpha_{\eta_{j+1}}$ by 
cutting the elements in $T_{\ell+1-j}$ into sequences as follows: $\beta_{\eta_j} = T_{\ell+1-j} \cap ( -\infty, b_{\eta_j+1} ]$, 
$\beta_{\eta_j+m} = T_{\ell+1-j} \cap (b_{\eta_j+m+1}, b_{\eta_j+m+2}]$ and $\alpha_{\eta_{j+1}} = T_{\ell+1-j} \cap 
(b_{\eta_{j+1}},+\infty)$. 

The inequalities are naturally forced from the inequalities in the semistandard tableaux, and the descents at the 
given positions are also forced,  because by construction $\alpha_{\eta_j} > b_{\eta_j} > b_{\eta_j+1} \geqslant \beta_{\eta_j}$. 
This process constructs the $b_i$, $\alpha_i$, and $\beta_i$ for each $i=1,\dots,k$, where we assume that sequences that 
have not been defined by the process are empty.    Then $\varphi^{-1}(T_1 \times T_2 \times \cdots \times T_{\ell+1})
= \pi = (\pi_1 \mid \pi_2 \mid \cdots \mid \pi_{k})$, where $\pi_i = b_i \alpha_i \beta_i$.  
\end{proof}
 
For a (skew) partition $\lambda$, the \defn{Schur function} $\mathsf{s}_\lambda(\mathbf{x})$ is defined as
\begin{equation}
\label{equation.schur}
	\mathsf{s}_\lambda(\mathbf{x}) = \sum_{T \in \mathsf{SSYT}(\lambda)} \mathbf{x}^{\wt(T)}.
\end{equation}
Similarly for $m \geqslant 1$, the \defn{$m$-th elementary symmetric function} $\mathsf{e}_m(\mathbf{x})$ is given by
\[
	\mathsf{e}_m(\mathbf{x}) = \sum_{1 \leqslant j_1 < j_2 < \cdots < j_m} x_{j_1} x_{j_2} \cdots x_{j_m}.
\]
As an immediate consequence of Proposition~\ref{P:biject}, we have the following symmetric function identity.
 
\begin{corollary}
\label{C:symexpand}   
Assume $\mathrm{D} \subseteq \{1,2,\dots,n-1\}$ and $\mathrm{I} \subseteq\{1,2,\dots, k-1\}$ are sets of $\ell$ 
distinct elements each and let $\mathrm{M(D,I)}$, $\gamma$ and $c_j$ for $1\leqslant j \leqslant \ell$ be as in 
Proposition~\ref{P:biject}. Then 
\[
	\sum_{\pi \in \mathrm{M(D,I)}} \mathbf{x}^{\wt(\pi)} 
	= \mathsf{s}_\gamma(\mathbf{x}) \ \prod_{j=1}^\ell \mathsf{e}_{c_j}(\mathbf{x}).
\]
\end{corollary}

\section{Crystal on ordered multiset partitions}
\label{section.crystal}

\subsection{Crystal structure}

Denote the set of words of length $n$ over the alphabet $\{1,2,\ldots,r\}$ by $\mathcal{W}^{(r)}_n$.
The set $\mathcal{W}_n^{(r)}$ can be endowed with an $\mathfrak{sl}_r$-crystal structure as follows.
The weight $\wt(w)$ of $w\in \mathcal{W}_n^{(r)}$ is the tuple $(a_1,\ldots,a_r)$, where $a_i$ is the number
of letters $i$ in $w$. The \defn{Kashiwara raising} and \defn{lowering operators}
\[
	e_i, f_i \colon \mathcal{W}_n^{(r)} \to \mathcal{W}_n^{(r)} \cup \{0\} \qquad \qquad \text{for $1\leqslant i <r$}
\]
are defined as follows. Associate to each letter $i$ in $w$ an open bracket ``$)$'' and to each letter $i+1$ in $w$
a closed bracket ``$($''. Then $e_i$ changes the $i+1$ associated to the leftmost unmatched ``$($'' to an $i$; if there 
is no such letter, $e_i(w)=0$. Similarly, $f_i$ changes the $i$ associated to the rightmost unmatched ``$)$'' to an $i+1$;
if there is no such letter, $f_i(w)=0$.

For $\lambda$ a (skew) partition, the $\mathfrak{sl}_r$-crystal action on $\mathsf{SSYT}^{(r)}(\lambda)$ is induced
by the crystal on $\mathcal{W}_{|\lambda|}^{(r)}$, where $|\lambda|$ is the number of boxes in $\lambda$. Consider the 
row-reading word $\mathsf{row}(T)$ of $T\in \mathsf{SSYT}^{(r)}(\lambda)$, which is the word obtained from $T$ by 
reading the rows from bottom to top, left to right. Then $f_i(T)$ (resp. $e_i(T)$) is the RSK insertion tableau of
$f_i(\mathsf{row}(T))$ (resp. $e_i(\mathsf{row}(T))$). It is well known that {$f_i(T)$ is a tableau in
$\mathsf{SSYT}^{(r)}(\lambda)$ with weight  equal to} $\wt(T)-\epsilon_i+\epsilon_{i+1}$, where $\epsilon_i$ is $i$-th 
standard vector in $\ZZ^r$.
Similarly, $e_i(T) \in \mathsf{SSYT}^{(r)}(\lambda)$, and $e_i(T)$ has weight $\wt(T)+\epsilon_i-\epsilon_{i+1}$. See for 
example~\cite[Chapter 3]{Bump.Schilling.2017}.

In the same spirit, an $\mathfrak{sl}_r$-crystal structure can be imposed on 
\[
	\mathsf{SSYT}^{(r)}(1^{c_1},\ldots,1^{c_\ell},\gamma)
	:= \mathsf{SSYT}^{(r)}(1^{c_1}) \times \cdots \times \mathsf{SSYT}^{(r)}(1^{c_\ell}) \times \mathsf{SSYT}^{(r)}(\gamma)
\]
by concatenating the reading words of the tableaux in the tuple. This yields crystal operators
\[
	e_i,f_i \colon \mathsf{SSYT}^{(r)}(1^{c_1},\ldots,1^{c_\ell},\gamma) \to 
	\mathsf{SSYT}^{(r)}(1^{c_1},\ldots,1^{c_\ell},\gamma) \cup \{0\}.
\]
Via the bijection $\varphi$ of Proposition~\ref{P:biject}, this also imposes crystal operators on ordered
multiset partitions
\[
	\tilde{e}_i,\tilde{f}_i \colon \mathcal{OP}_{n,k}^{(r)} \to \mathcal{OP}_{n,k}^{(r)} \cup \{0\}
\]
as $\tilde{e}_i = \varphi^{-1} \circ e_i \circ \varphi$ and $\tilde{f}_i = \varphi^{-1} \circ f_i \circ \varphi$.

An example of a crystal structure on $\mathcal{OP}_{n,k}^{(r)}$ is given in Figure~\ref{figure.crystal}.

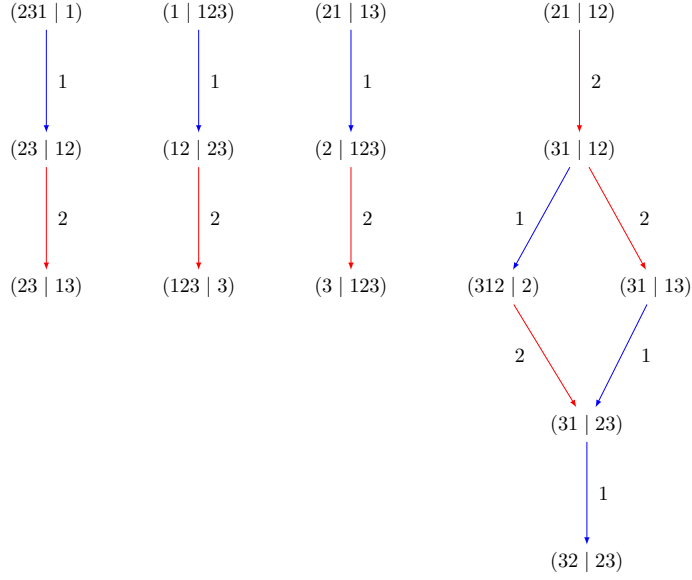
\begin{figure}
\scalebox{0.7}{
\begin{tikzpicture}[>=latex,line join=bevel,]
\node (node_13) at (323.0bp,9.0bp) [draw,draw=none] {$\left(32 \mid 23\right)$};
  \node (node_14) at (114.0bp,305.0bp) [draw,draw=none] {$\left(1 \mid 123\right)$};
  \node (node_9) at (319.0bp,231.0bp) [draw,draw=none] {$\left(31\mid 12\right)$};
  \node (node_8) at (32.0bp,231.0bp) [draw,draw=none] {$\left(23 \mid 12\right)$};
  \node (node_7) at (360.0bp,157.0bp) [draw,draw=none] {$\left(31 \mid 13\right)$};
  \node (node_6) at (196.0bp,231.0bp) [draw,draw=none] {$\left(2\mid 123\right)$};
  \node (node_5) at (278.0bp,157.0bp) [draw,draw=none] {$\left(312\mid 2\right)$};
  \node (node_4) at (114.0bp,231.0bp) [draw,draw=none] {$\left(12\mid 23\right)$};
  \node (node_3) at (32.0bp,157.0bp) [draw,draw=none] {$\left(23 \mid 13\right)$};
  \node (node_2) at (196.0bp,157.0bp) [draw,draw=none] {$\left(3\mid 123\right)$};
  \node (node_1) at (114.0bp,157.0bp) [draw,draw=none] {$\left(123 \mid 3\right)$};
  \node (node_0) at (32.0bp,305.0bp) [draw,draw=none] {$\left(231\mid 1\right)$};
  \node (node_11) at (319.0bp,305.0bp) [draw,draw=none] {$\left(21 \mid 12\right)$};
  \node (node_10) at (323.0bp,83.0bp) [draw,draw=none] {$\left(31\mid 23\right)$};
  \node (node_12) at (196.0bp,305.0bp) [draw,draw=none] {$\left(21\mid 13\right)$};
  \draw [blue,->] (node_10) ..controls (323.0bp,62.872bp) and (323.0bp,42.801bp)  .. (node_13);
  \definecolor{strokecol}{rgb}{0.0,0.0,0.0};
  \pgfsetstrokecolor{strokecol}
  \draw (332.0bp,46.0bp) node {$1$};
  \draw [red,->] (node_6) ..controls (196.0bp,210.87bp) and (196.0bp,190.8bp)  .. (node_2);
  \draw (205.0bp,194.0bp) node {$2$};
  \draw [red,->] (node_9) ..controls (330.34bp,210.54bp) and (341.93bp,189.61bp)  .. (node_7);
  \draw (354.0bp,194.0bp) node {$2$};
  \draw [red,->] (node_8) ..controls (32.0bp,210.87bp) and (32.0bp,190.8bp)  .. (node_3);
  \draw (41.0bp,194.0bp) node {$2$};
  \draw [blue,->] (node_9) ..controls (307.66bp,210.54bp) and (296.07bp,189.61bp)  .. (node_5);
  \draw (287.0bp,194.0bp) node {$1$};
  \draw [red,->] (node_11) ..controls (319.0bp,284.87bp) and (319.0bp,264.8bp)  .. (node_9);
  \draw (328.0bp,268.0bp) node {$2$};
  \draw [red,->] (node_5) ..controls (290.44bp,136.54bp) and (303.17bp,115.61bp)  .. (node_10);
  \draw (287.0bp,120.0bp) node {$2$};
  \draw [blue,->] (node_0) ..controls (32.0bp,284.87bp) and (32.0bp,264.8bp)  .. (node_8);
  \draw (41.0bp,268.0bp) node {$1$};
  \draw [red,->] (node_4) ..controls (114.0bp,210.87bp) and (114.0bp,190.8bp)  .. (node_1);
  \draw (123.0bp,194.0bp) node {$2$};
  \draw [blue,->] (node_12) ..controls (196.0bp,284.87bp) and (196.0bp,264.8bp)  .. (node_6);
  \draw (205.0bp,268.0bp) node {$1$};
  \draw [blue,->] (node_7) ..controls (349.82bp,136.65bp) and (339.51bp,116.01bp)  .. (node_10);
  \draw (355.0bp,120.0bp) node {$1$};
  \draw [blue,->] (node_14) ..controls (114.0bp,284.87bp) and (114.0bp,264.8bp)  .. (node_4);
  \draw (123.0bp,268.0bp) node {$1$};
\end{tikzpicture}
}
\caption{The crystal structure on $\mathcal{OP}_{4,2}^{(3)}$. The $\mm$ of the connected components
are $2,0,1,1$ from left to right.
\label{figure.crystal}}
\end{figure}

\begin{theorem}
\label{theorem.crystal}
The operators $\tilde{e}_i, \tilde{f}_i$, and $\wt$ impose an $\mathfrak{sl}_r$-crystal structure on
$\mathcal{OP}_{n,k}^{(r)}$. In addition, $\tilde{e}_i$ and $\tilde{f}_i$ preserve the $\mm$ statistic.
\end{theorem}

\begin{proof}
The operators $\tilde{e}_i, \tilde{f}_i$, and $\wt$ impose an $\mathfrak{sl}_r$-crystal structure by construction
since $\varphi$ is a weight-preserving bijection. The Kashiwara operators $\tilde{e}_i$ and $\tilde{f}_i$ preserve
the $\mm$ statistic, since by Proposition~\ref{P:biject}, the bijection $\varphi$ restricts to $\mathrm{M(D,I)}^{(r)}$ 
which fixes the descents of the ordered multiset partitions in minimaj order.
\end{proof}

\subsection{Explicit crystal operators}

Let us now write down the crystal operator $\tilde{f}_i \colon \mathcal{OP}_{n,k} \to \mathcal{OP}_{n,k}$ 
of Theorem~\ref{theorem.crystal} explicitly on $\pi\in \mathcal{OP}_{n,k}$ in minimaj order.

Start by creating a word $w$ from right to left by reading the first element in each block of $\pi$ from right to left, 
followed by the remaining elements of $\pi$ from left to right. Note that this agrees with $\mathsf{row}(\varphi(\pi))$.
For example, $w=513165421434212$ for $\pi$ in Example~\ref{example.pi phi}.
Use the crystal operator $f_i$ on words to determine 
which $i$ in $w$ to change to an $i+1$. Circle the corresponding letter $i$ in $\pi$. The crystal operator $\tilde{f}_i$ 
on $\pi$ changes the circled $i$ to $i+1$ unless we are in one of the following two cases:
\begin{subequations}
\label{equation.f explicit}
\begin{align}
\label{equation.f explicit a}
    	\cdots \encircle{i} \mid i &\quad \stackrel{\tilde{f}_i}{\longrightarrow} \quad \cdots \mid i \; \encircle{i\!\!+\!\!1} \;, \\
\label{equation.f explicit b}
	\mid \encircle{i} \; i\!\!+\!\!1 & \quad \stackrel{\tilde{f}_i}{\longrightarrow} \quad i\!\!+\!\!1 \mid \encircle{i\!\!+\!\!1} \;.
\end{align}
\end{subequations}
Here ``$\cdots$'' indicates that the block is not empty in this region.

\begin{example}
In Figure~\ref{figure.crystal}, $\tilde{f}_2(31 \encircle{2} \mid 2) = (31 \mid 2 \; \encircle{3})$ is an example
of~\eqref{equation.f explicit a}. Similarly, $\tilde{f}_1(31 \mid \encircle{1}\;2) = (312 \mid \encircle{2})$ is an example
of~\eqref{equation.f explicit b}.
\end{example}

\begin{proposition}
The above explicit description for $\tilde{f}_i$ is well defined and agrees with the definition of
Theorem~\ref{theorem.crystal}.
\end{proposition}

\begin{proof}
The word $w$ described above is precisely $\mathsf{row}(\varphi(\pi))$ on which $f_i$ acts.
Hence the circled letter $i$ is indeed the letter changed to $i+1$. It remains to check how $\varphi^{-1}$
changes the blocks. We will demonstrate this for the cases in~\eqref{equation.f explicit} as the other cases are similar.

In case \eqref{equation.f explicit a} the circled letter $i$ in block $\pi_j$ does not correspond to $b_j$ in $\pi_j$ as it is not at
the beginning of its block. Hence, it belongs to $\alpha_j$ or $\beta_j$. The circled letter is not a descent.
Changing it to $i+1$ would create a descent. The map $\varphi^{-1}$ distributes the letters in $\alpha_j$
and $\beta_j$ to preserve descents, hence the circled $i$ moves over to the next block on the right and becomes a 
circled $i+1$. Note also that $i+1 \not \in \pi_{j+1}$, since otherwise the circled $i$ would have been bracketed in $w$, 
contradicting the fact that $f_i$ is acting on it.

In case \eqref{equation.f explicit b} the circled letter $i$ in block $\pi_j$ corresponds to $b_j$ in $\pi_j$. Again,
$\varphi^{-1}$ now associates the $i+1 \in \pi_j$ to the previous block after applying $f_i$. Note that 
$i+1 \not \in \pi_{j-1}$ since it would necessarily be $b_{j-1}$. But then the circled $i$ would have been bracketed
in $w$, contradicting the fact that $f_i$ is acting on it.
\end{proof}

\subsection{Schur expansion}

The \defn{character} of an $\mathfrak{sl}_r$-crystal $B$ is defined as
\[
	\mathrm{ch} B = \sum_{b\in B} \mathbf{x}^{\wt(b)}.
\]
Denote by $B(\lambda)$ the $\mathfrak{sl}_\infty$-crystal on $\mathsf{SSYT}(\lambda)$ defined above.
This is a connected highest weight crystal with highest weight $\lambda$, and the character is the Schur 
function $\mathsf{s}_\lambda(\mathbf{x})$ defined in~\eqref{equation.schur}
\[
	\mathrm{ch}B(\lambda) = \mathsf{s}_\lambda(\mathbf{x}).
\]
Similarly, denoting by $B^{(r)}(\lambda)$ the $\mathfrak{sl}_r$-crystal on $\mathsf{SSYT}^{(r)}(\lambda)$,
its character is the Schur polynomial
\[
	\mathrm{ch}B^{(r)}(\lambda) = \mathsf{s}_\lambda(x_1,\ldots,x_r).
\]
Let us define
\[
	\mathsf{Val}^{(r)}_{n,k}(\mathbf{x};0,t) = \sum_{\pi\in \mathcal{OP}^{(r)}_{n,k+1}}  t^{\mm(\pi)} \mathbf{x}^{\wt(\pi)},
\]
which satisfies $\mathsf{Val}_{n,k}(\mathbf{x};0,t) = \mathsf{Val}_{n,k}^{(r)}(\mathbf{x};0,t)$ for $r\geqslant n$,
where $\mathsf{Val}_{n,k}(\mathbf{x};0,t)$ is as in~\eqref{equation.val}.

As a consequence of Theorem~\ref{theorem.crystal}, we now obtain the Schur expansion of 
$\mathsf{Val}_{n,k}^{(r)}(\mathbf{x};0,t)$.

\begin{corollary}
We have
\[
	\mathsf{Val}_{n,k-1}^{(r)}(\mathbf{x};0,t) 
	= \sum_{\substack{\pi \in \mathcal{OP}^{(r)}_{n,k}\\ \tilde{e}_i(\pi) = 0 \;\; \forall \;1\leqslant i <r}} t^{\mm(\pi)} 
	\mathsf{s}_{\wt(\pi)}.
\]
\end{corollary}
When $r\geqslant n$, then by~\cite{Wilson.2016} and~\cite[Proposition 3.18]{Rhoades.2016} this is also equal to
\[
	\mathsf{Val}_{n,k-1}(\mathbf{x};0,t) = \sum_{\lambda \vdash n} \;\; \sum_{T \in \mathsf{SYT}(\lambda)}
	t^{\maj(T) + \binom{n-k}{2} -(n-k) \mathsf{des}(T)} \left[ \begin{array}{c} \mathsf{des}(T)\\ n-k \end{array} \right]
	\mathsf{s}_\lambda(\mathbf{x}),
\]
where $\mathsf{SYT}(\lambda)$ is the set of standard Young tableaux of shape $\lambda$ (that is, the elements in
$\mathsf{SSYT}(\lambda)$ of weight $(1^{|\lambda|})$), $\mathsf{des}(T)$ is the number of descents of $T$,
$\maj(T)$ is the major index of $T$ (or the sum of descents of $T$), and the $t$-binomial coefficients in the sum
are defined using the rule
\[
	 \left[ \begin{array}{c} m \\ p \end{array} \right] = \frac{[m]!}{[p]!\ [m-p]!} \ \ \text{where $[p]! = [p][p-1] \cdots [2][1]$
	and \  $[p] = 1 + t + \cdots + t^{p-1}$}.
\]

\begin{example}
The crystal $\mathcal{OP}_{4,2}^{(3)}$, displayed in Figure~\ref{figure.crystal}, has four highest weight elements
with weights $(2,1,1)$, $(2,1,1)$, $(2,1,1)$, $(2,2)$ from left to right. Hence, we obtain the Schur expansion
\[
	\mathsf{Val}^{(3)}_{4,1}(\mathbf{x};0,t) = (1+t+t^2)\; \mathsf{s}_{(2,1,1)}(\mathbf{x}) + t \;\mathsf{s}_{(2,2)}(\mathbf{x}).
\]
\end{example}

\section{Equidistributivity of the minimaj and maj statistics}
\label{section.equi}

In this section, we describe a bijection $\psi \colon \calOP_{n,k} \to \calOP_{n,k}$ in Theorem~\ref{theorem.bij OP}
with the property that $\mm(\pi) = \maj(\psi(\pi))$ for $\pi \in \calOP_{n,k}$. This proves the link between
$\mm$ and $\maj$ that was missing in~\cite{Wilson.2016}.
We can interpret $\psi$ as a crystal isomorphism, where $\calOP_{n,k}$ on the left is the $\mm$ crystal of
Section~\ref{section.crystal} and $\calOP_{n,k}$ on the right is viewed as a crystal of $k$ columns with elements
written in major index order.

The bijection $\psi$ is the composition of $\varphi$ of Proposition~\ref{P:biject} with a certain shift operator.
When applying $\varphi$ to $\pi \in \calOP_{n,k}$, we obtain the tuple $T^\bullet=T_1 \times \cdots \times T_{\ell+1}$ 
in~\eqref{equation.T picture}.
We would like to view each column in the tuple of tableaux as a block of a new ordered multiset partition. However, note
that some columns could be empty, namely if $c_j=d_{\ell+2-j}-i_{\ell+2-j}$ in Proposition~\ref{P:biject} is zero for some 
$1\leqslant j \leqslant \ell$. For this reason, let us introduce the set of \defn{weak ordered multiset partitions} 
$\mathcal{WOP}_{n,k}$, where we relax the condition that all blocks need to be nonempty sets.

Let $T^\bullet = T_1 \times \cdots \times T_{\ell+1}$ be a tuple of skew tableaux. Define $\read(T^\bullet)$ to be the 
weak ordered multiset partition whose blocks are obtained from $T^\bullet$ by reading the columns from the 
left to the right and from the bottom to the top; each column constitutes one of the blocks in $\read(T^\bullet)$.
Note that given $\pi= (\pi_1 | \pi_2| \cdots | \pi_k) \in \calOP_{n,k}$ in minimaj order, $\read(\varphi(\pi))$ is a weak
ordered multiset partition in major index order.

\begin{example} 
\label{example.pi ex}
Let $\pi = (1\mid 56.\mid 4.\mid 37.12\mid 2.1\mid 1\mid 34) \in \calOP_{13,7}$, written in minimaj order. 
We have $\mm(\pi)=22$. Then 
\ytableausetup{boxsize=1.1em}
\[
	T^\bullet = \varphi(\pi) =  
	\ytableaushort{1,4} \times  \ytableaushort{1,2} 
	\times \ytableaushort{7} \times \emptyset \times 
	\ytableaushort{\none 13,\none 2,\none 3,\none 4,15,6}
\]
and $\pi'=\read(T^\bullet) = (4.1\mid 2.1\mid 7.\mid \emptyset \mid 6.1\mid 5.4.3.2.1 \mid 3)$.
\end{example}

\begin{lemma}
\label{lem.majproperties}
Let $\mathcal{I}=\{\read(\varphi(\pi)) \mid \pi \in \calOP_{n,k}\} \subseteq \mathcal{WOP}_{n,k}$,
$\pi' = \read(\varphi(\pi)) \in \mathcal{I}$, and $b_i$ the first elements in each block of $\pi$ in minimaj order as
in Lemma~\ref{lemma.minimaj order}. Then $\pi'$ has the following properties:  
\begin{enumerate}
\item The last $k$ elements of $\pi'$ are $b_1,\ldots,b_k$, and $b_i$ and $b_{i+1}$ are in different blocks if and only 
if $b_i \leqslant b_{i+1}$.
\item If $b_1,\ldots,b_k$ are contained in precisely $k-j$ blocks, then there are at least $j$ descents in the blocks 
containing the $b_i$'s. 
\end{enumerate}
\end{lemma}

\begin{proof}
Let $\pi\in \calOP_{n,k}$, written in minimaj order. Then by~\eqref{equation.T picture}, 
$\pi'=\read(\varphi(\pi))$ is of the form
\[
	\pi'=
	(\alpha^{\mathrm{rev}}_{\eta_{\ell+1}}\beta^{\mathrm{rev}}_{\eta_{\ell+1}-1}\cdots\beta^{\mathrm{rev}}_{\eta_\ell} \mid
	\cdots\mid
	\alpha^{\mathrm{rev}}_{\eta_1}\beta^{\mathrm{rev}}_{\eta_1-1}\cdots\beta^{\mathrm{rev}}_1b_1 \cdots \mid 
	\cdots \mid 
	b_{\eta_1}b_{\eta_1-1} \cdots \mid 
	\cdots \mid 
	\cdots b_k),
\]
where the superscript $\mathrm{rev}$ indicates that the elements are listed in decreasing order (rather than increasing order).
Since the rows of a semistandard tableau are weakly increasing and the columns are strictly increasing, the blocks 
of $\pi'=\read(\varphi(\pi))$ are empty or in strictly decreasing order. This implies that $b_i$ and $b_{i+1}$ are in different 
blocks of $\pi'$ precisely when $b_i\leqslant b_{i+1}$, so a block of $\pi'$ that contains a $b_i$ cannot have a descent 
at its end. This proves~(1).

In a weak ordered multiset partition written in major index order, any block of size $r\geqslant 2$ has $r-1$ descents. 
So if $b_1,\ldots, b_k$ are contained in precisely $k-j$ blocks, then at least $j$ of these elements are contained in blocks 
of size at least two, so there are at least $j$ descents in the blocks containing the $b_i$'s. This proves~(2).
\end{proof}

\begin{remark}
Let $\pi'\in \mathcal{WOP}_{n,k}$ be in major index order such that 
there are at least $k$ elements after the rightmost occurrence of a block that is either empty or has a descent at its end.  
In this case, there exists a skew tableau $T^\bullet$ such that $\pi'=\read(T^\bullet)$. In fact, this characterizes 
$\mathcal{I} := \mathrm{im} (\read \circ \varphi)$.
\end{remark}

\begin{lemma}
\label{lemma.read}
The map $\read$ is invertible.
\end{lemma}

\begin{proof}
Suppose $\pi' \in \mathcal{WOP}_{n,k}$ is in major index order such that there are at least $k$ elements after the 
rightmost occurrence of a block that is either empty or has a descent at its end. Since there are no occurrences of an 
empty block or a descent at the end of a block amongst the last $k$ elements of $\pi'$, the blocks of $\pi'$ containing 
the last $k$ elements form the columns of a skew ribbon tableau $T\in \mathsf{SSYT}(\gamma)$, and the remaining blocks 
of $\pi'$ form the column tableaux to the left of the skew ribbon tableau, so $\read$ is invertible.
\end{proof}

We are now ready to introduce the shift operators.

\begin{definition}
\label{definition.Lshift}
We define the \defn{left shift operation} $\L$ on $\pi'\in \mathcal{I} = \{ \read(\varphi(\pi)) \mid \pi \in \mathcal{OP}_{n,k}\}$ 
as follows. Suppose $\pi'$ has $m \geqslant 0$ blocks $\pi_{p_m}',\ldots, \pi_{p_1}'$ that are either empty or have 
a descent at the end, and $1 \leqslant p_m < \cdots <  p_2 < p_1<k$. Set
\[
	\L(\pi') = \L^{(m)}(\pi'),
\]
where $\L^{(i)}$ for $0\leqslant i\leqslant m$ are defined as follows:
\begin{enumerate}
\item
Set $\L^{(0)}(\pi')=\pi'$.
\item
Suppose $\L^{(i-1)}(\pi')$ for $1\leqslant i \leqslant m$ is defined. By induction, the $p_i$-th block of $\L^{(i-1)}(\pi')$ is
$\pi'_{p_i}$. Let $S_i$ be the sequence of elements starting immediately to the right of block $\pi'_{p_i}$ in $\L^{(i-1)}(\pi')$ 
up to and including the $p_i$-th descent after the block $\pi_{p_i}'$. Let $\L^{(i)}(\pi')$ be the weak ordered multiset 
partition obtained by moving each element in $S_i$ one block to its left. Note that all blocks with index smaller than
$p_i$ in $\L^{(i)}(\pi')$ are the same as in $\pi'$.
\end{enumerate}
\end{definition}

\begin{example}
\label{example.pi ex2}
Continuing Example~\ref{example.pi ex}, we have 
$\pi'=(4.1\mid 2.1\mid 7.\mid \emptyset \mid {\color{blue}6}.{\color{blue}1}\mid {\color{blue}5}.{\color{blue}4}.{\color{blue}3}
.2.1 \mid 3)$,
which is in major index order. We have $m=2$ with $p_2=3<4=p_1$, $S_1={\color{blue}61543}$, 
$S_2={\color{red}6154}$ and
\begin{equation*}
\begin{split}
	\L^{(1)}(\pi') &= (4.1\mid 2.1\mid 7. \mid {\color{red}6}.{\color{red}1}\mid {\color{red}5}.{\color{red}4}.3.\mid 2.1\mid 3),\\
	\L(\pi') = \L^{(2)}(\pi') &= (4.1\mid 2.1\mid 7.6.1 \mid 5.4.\mid 3.\mid 2.1\mid 3).
\end{split}
\end{equation*}
Note that $\maj(\pi')=28$, $\maj(\L^{(1)}(\pi'))=25$, and $\maj(\L(\pi')) = 22 = \mm(\pi)$.
\end{example}

\begin{proposition}
\label{proposition.L}
The left shift operation $\L \colon \mathcal{I} \to \mathcal{OP}_{n,k}$ is well defined.
\end{proposition}

\begin{proof}
Suppose $\pi'\in\mathcal{I}$ has $m \geqslant 0$ blocks $\pi_{p_1}',\ldots, \pi_{p_m}'$ that are either empty or have a 
descent at the end, and $1\leqslant p_m < \cdots < p_2 < p_1 < k$.
If $m=0$, then $\L(\pi')=\pi' \in \mathcal{OP}_{n,k}$ and we are done.

We proceed by induction on $m$. Note that $\L^{(1)}$ acts on the rightmost block $\pi_{p_1}'$.  
Notice that $\pi_{p_1}'$ cannot contain any of the $b_i$'s by Lemma~\ref{lem.majproperties}~(1).
Hence, since there are at least $k$ elements in the $k-p_1$ blocks following $\pi_{p_1}'$, 
by Lemma~\ref{lem.majproperties}~(2), there are at least $p_1$ descents after $\pi_{p_1}'$, so $\L^{(1)}$ can be applied 
to $\pi'$.

Observe that applying $\L^{(1)}$ to $\pi'$ does not create any new empty blocks to the right of $\pi_{p_1}'$, because 
creating a new empty block means that the last element of $S_1$, which is a descent, is at the end of a block. 
This cannot happen, since the rightmost occurrence of an empty block or a descent at the end of its block was assumed 
to be in $\pi_{p_1}'$.  However, note that applying $\L^{(1)}$ to $\pi'$ does create a new block with a descent at its end, 
and this descent is given by the $p_1$-th descent after the block $\pi_{p_1}'$ (which is the last element of $S_1$).

Now suppose $\L^{(i-1)}(\pi')$ is defined for $i \geqslant 2$. By induction, there are at least $p_1>p_i$ descents following 
the block $\pi_{p_i}'$, so the set $S_i$ of Definition~\ref{definition.Lshift} exists and we can move the elements in $S_i$
left one block to construct $\L^{(i)}(\pi')$ from $\L^{(i-1)}(\pi')$. Furthermore, $\L^{(i)}(\pi')$ does not have any 
new empty blocks to the right of $\pi_{p_i}'$. To see this, note that the number of descents in $S_i$ is $p_i$, so the 
number of descents in $S_i$ is strictly decreasing as $i$ increases. This implies that the $i-1$ newly created descents 
at the end of a block of $\L^{(i-1)}(\pi')$ occurs strictly to the right of $S_i$, and so the last element of $S_i$ cannot 
be a descent at the end of a block of $\L^{(i-1)}(\pi')$.

Lastly, $\L(\pi') = \L^{(m)}(\pi')\in \calOP_{n,k}$, since it does not have any empty blocks, and every block of 
$\L(\pi')$ is in decreasing order because either we moved every element of a block into an empty block or we moved 
elements into a block with a descent at the end.
\end{proof}

\begin{definition}
\label{definition.Rshift}
We define the \defn{right shift operation} $\R$ on $\mu\in \mathcal{OP}_{n,k}$ in major index order as follows.  
Suppose $\mu$ has $m\geqslant 0$ blocks  $\mu_{q_1}, \ldots, \mu_{q_m}$ that have a descent at the end and 
$q_1 < q_2 < \cdots < q_m$. Set
\[
	\R(\mu) = \R^{(m)}(\mu),
\]
where $\R^{(i)}$ for $0\leqslant i \leqslant m$ are defined as follows:
\begin{enumerate}
\item
Set $\R^{(0)}(\mu)=\mu$.
\item
Suppose $\R^{(i-1)}(\mu)$ for $1\leqslant i \leqslant m$ is defined. Let $U_i$ be the sequence of $q_i$ elements 
to the left of, and including, the last element in the $q_i$-th block of $\R^{(i-1)}(\mu)$.  Let $\R^{(i)}(\mu)$ be the 
weak ordered multiset partition obtained by moving each element in $U_i$ one block to its right.
Note that all blocks to the right of the $(q_i+1)$-th block are the same in $\mu$ and $\R^{(i)}(\mu)$.
\end{enumerate}
\end{definition}

Note that $\R$ can potentially create empty blocks.

\begin{example}
Continuing Example~\ref{example.pi ex2}, let $\mu = \L(\pi') = (4.1\mid 2.1\mid 7.6.1 \mid 5.4.\mid 3.\mid 2.1\mid 3)$.  
We have $m=2$ with $q_1=4<5=q_2$, $U_1=6154$, $U_2=61543$ and
\begin{equation*}
\begin{split}
	\R^{(1)}(\mu) &= (4.1\mid 2.1\mid 7. \mid 6.1 \mid 5.4.3.\mid 2.1\mid 3),\\
	\R(\mu) = \R^{(2)}(\mu) &= (4.1\mid 2.1\mid 7. \mid \emptyset \mid 6.1 \mid 5.4.3.2.1\mid 3),
\end{split}
\end{equation*}
which is the same as $\pi'$ in Example~\ref{example.pi ex2}.
\end{example}

\begin{proposition}
\label{proposition.R}
The right shift operation $\R$ is well defined and is the inverse of $\L$.
\end{proposition}

\begin{proof}
Suppose $\mu\in \mathcal{OP}_{n,k}$ in major index order has descents at the end of the blocks 
$\mu_{q_1},\ldots, \mu_{q_m}$. If $m=0$, then $\R(\mu) = \mu \in\mathcal{OP}_{n,k} \subseteq \mathcal{WOP}_{n,k}$ 
and there is nothing to show.

We proceed by induction on $m$. The ordered multiset partition $\mu$ does not have empty blocks, so there are at 
least $q_1$ elements in the first $q_1$ blocks of $\mu$, and $\R^{(1)}$ can be applied to $\mu$.

Now suppose $\R^{(i-1)}(\mu)$ is defined for $i\geqslant2$. By induction, there are at least $q_{i-1}+1$ elements in 
the first $q_{i-1}+1$ blocks of $\R^{(i-1)}(\mu)$.  Since the blocks $\mu_{q_{i-1}+2},\ldots, \mu_{q_i}$ in $\mu$ 
are all nonempty, there are at least $q_{i-1}+1+(q_i-(q_{i-1}+1)) = q_i$ elements in the first $q_i$ blocks of $\R^{(i-1)}(\mu)$, 
so the set $U_i$ of Definition~\ref{definition.Rshift} exists and we can move the elements in $U_i$ one block to the 
right to construct $\R^{(i)}(\mu)$ from $\R^{(i-1)}(\mu)$.  

Furthermore, every nonempty block of $\R(\mu)$ is in decreasing order because the rightmost element of each $U_i$ is 
a descent. So $\R(\mu)\in\mathcal{OP}_{n,k}$ remains in major index order. This completes the proof that $\R$
is well defined.

Next we show that $\R$ is the inverse of $\L$. Observe that if $\pi' \in \mathcal{I}$ has $m$ occurrences of either an 
empty block or a block with a descent at its end, then $\mu=\L(\pi')$ has $m$ blocks with a descent at its end.  
Hence it suffices to show that $\R^{(m+1-i)}$ is the inverse operation to $\L^{(i)}$ for each $1\leqslant i \leqslant m$.

The property that the last element of $S_i$ cannot be a descent at the end of a block of $\L^{(i-1)}(\pi')$ in the proof of
Proposition~\ref{proposition.L} similarly holds for every element in $S_i$. Therefore, if the last element of $S_i$ is in the 
$r_i$-th block of $\L^{(i-1)}(\pi')$, then $|S_i| = p_i + (r_i-1-p_i) = r_i-1$ because the blocks are decreasing and none of the 
elements in $S_i$ can be descents at the end of a block.
Since the last element of $S_i$ becomes a descent at the end of the $(r_i-1)$-th block of $\L^{(i)}(\pi)$, this implies 
$r_i-1 = q_{m-i+1}$, so $U_{m-i+1} = S_i$ for every $1\leqslant i \leqslant m$. As the operation $\L^{(i)}$ is a left shift 
of the elements of $S_i$ by one block and the operation $\R^{(m+1-i)}$ is a right shift of the same set of elements 
by one block, they are inverse operations of each other.
\end{proof}

For what follows, we need to extend the definition of the major index to the set $\mathcal{WOP}_{n,k}$ of weak ordered 
multiset partitions of length $n$ and $k$ blocks, in which some of the blocks may be empty. Given 
$\pi' \in \mathcal{WOP}_{n,k}$ whose nonempty blocks are in major index order, if the block $\pi_j'\neq\emptyset$, then 
the last element in $\pi_j'$ is assigned the index $j$, and the remaining elements in $\pi_j'$ are assigned the index $j-1$
for $j=1,\ldots, k$. Then $\maj(\pi')$ is the sum of the indices where a descent occurs. This agrees 
with~\eqref{equation.maj} in the case when all blocks are nonempty.

\begin{lemma}
\label{lemma.maj change}
Let $\pi'\in \mathcal{I}$. With the same notation as in Definition~\ref{definition.Lshift},
we have for $1\leqslant i \leqslant m$
\[
	\maj(\L^{(i)}(\pi')) = \begin{cases}
	\maj(\L^{(i-1)}(\pi'))-p_i+1, & \text{if $\pi_{p_i}'=\emptyset$,}\\
	\maj(\L^{(i-1)}(\pi'))-p_i, &\text{if $\pi_{p_i}'$ has a descent at the end of its block}.
	\end{cases}
\]
\end{lemma}

\begin{proof} 
Assume $\pi_{p_i}'=\emptyset$. In the transformation from $\L^{(i-1)}(\pi')$ to $\L^{(i)}(\pi')$, the index of each of the first 
$p_i-1$ descents in $S_i$ decreases by one, while the index of the last descent remains the same, since it is not at the 
end of a block in $\L^{(i-1)}(\pi')$, but it becomes the last element of a block in $\L^{(i)}(\pi')$. The indices of elements 
not in $S_i$ remain the same, so $\maj(\L^{(i)}(\pi'))=\maj(\L^{(i-1)}(\pi'))-p_i+1$ in this case.

Next assume that $\pi_{p_i}'$ has a descent at the end of the block. In the transformation from $\L^{(i-1)}(\pi')$ to $\L^{(i)}(\pi')$, 
the indices of the descents in $S_i$ change in the same way as in the previous case, but in addition, the index of the last 
descent in $\pi_{p_i}'$ decreases by one, so $\maj(\L^{(i)}(\pi'))=\maj(\L^{(i-1)}(\pi'))-p_i$ in this case.
\end{proof}

\begin{theorem} 
\label{theorem.bij OP}
Let $\psi \colon \calOP_{n,k}\rightarrow \calOP_{n,k}$ be the map defined by
\[
	\psi(\pi) = \L(\read(\varphi(\pi))) \qquad \text{for $\pi\in \calOP_{n,k}$ in minimaj order.}
\]
Then $\psi$ is a bijection that maps ordered multiset partitions in minimaj order to ordered multiset partitions in 
major index order. Furthermore, $\mm(\pi) = \maj(\psi(\pi))$.
\end{theorem}

\begin{proof} 
By Proposition~\ref{P:biject}, $\varphi$ is a bijection. By Lemma~\ref{lemma.read}, the map $\read$ is invertible, and
by Proposition~\ref{proposition.R} the shift operation $\L$ has an inverse. This implies that $\psi$ is a bijection.

It remains to show that $\mm(\pi) = \maj(\psi(\pi))$ for $\pi \in \mathcal{OP}_{n,k}$ in minimaj order.

First suppose that $\pi' = \read(\varphi(\pi))$ has no empty blocks and no descents at the end of any block.
In this case $\L(\pi')=\pi'$, so that in fact $\pi' = \psi(\pi)$. Using the definition of major index~\eqref{equation.maj} and 
the representation~\eqref{equation.T picture} (where the columns in the ribbon are viewed as separate columns due 
to $\read$), we obtain
\begin{equation}
\label{equation.base maj}
	\maj(\pi') = \sum_{j=1}^\ell (\ell+1-j) ( d_j - i_j -1) + \ell + \sum_{j=1}^\ell ( \ell+\eta_j-j),
\end{equation}
where $d_j,i_j,\eta_j = i_1 + \cdots + i_j$ are defined in Proposition~\ref{P:biject} for $\pi$.
Here, the first sum in the formula arises from the contributions of the first $\ell$ blocks and the summand $\ell$ 
compensates for the fact that $b_1$ is in the $\ell$-th block. The second sum in the formula comes from the 
contributions of the $b_i$'s. Comparing with~\eqref{equation.minimaj}, we find
\[
	\maj(\pi') = \mm(\pi) - \binom{\ell+1}{2} - \sum_{j=1}^\ell (\ell+1-j) i_j + \binom{\ell+1}{2} 
	+ \sum_{j=1}^\ell \eta_j = \mm(\pi),
\]
proving the claim.

Now suppose that $\pi' = \read(\varphi(\pi))$ has a descent at the end of block $\pi'_p$. This will contribute an extra
$p$ compared to the major index in~\eqref{equation.base maj}. If $\pi'_p=\emptyset$, then
$c_p = d_{\ell+2-p} - i_{\ell+2-p} = 0$ and the term $j=\ell+2-p$ in~\eqref{equation.base maj} should be
$(\ell+1-j)(d_j-i_j)$ instead of $(\ell+1-j)(d_j-i_j-1)$ yielding a correction term of $\ell+1-j = \ell+1-\ell-2+p=p-1$.
Hence, with the notation of Definition~\ref{definition.Lshift}, we have
\[
	\maj(\pi') = \mm(\pi) + \sum_{i=1}^m p_i - e,
\]
where $e$ is the number of empty blocks in $\pi'$. Since $\psi(\pi) = \L(\pi')$, the claim follows by 
Lemma~\ref{lemma.maj change}.
\end{proof}

\bibliographystyle{alpha}
\bibliography{paper}{}

\end{document}